\documentclass{lms}
\usepackage{amsfonts}
\usepackage{amsmath,amssymb,amsxtra}
\usepackage{float}
\usepackage[colorlinks, linkcolor=RoyalBlue,anchorcolor=Periwinkle,
    citecolor=Orange,urlcolor=Green]{hyperref}
\usepackage[usenames,dvipsnames]{xcolor}
\usepackage{enumitem}
\usepackage{geometry} 
\usepackage{graphicx}
\usepackage{subfigure}
\usepackage{tikz}\usetikzlibrary{matrix}
\usepackage{url}
\usepackage[all]{xy}
\usepackage{booktabs,bookmark}

\newtheorem{theorem}{Theorem}[section] 
\newtheorem{lemma}[theorem]{Lemma}     
\newtheorem{corollary}[theorem]{Corollary}
\newtheorem{proposition}[theorem]{Proposition}
\newnumbered{assertion}{Assertion}    
\newnumbered{conjecture}{Conjecture}  
\newnumbered{definition}{Definition}
\newnumbered{hypothesis}{Hypothesis}
\newnumbered{remark}{Remark}
\newnumbered{note}{Note}
\newnumbered{observation}{Observation}
\newnumbered{problem}{Problem}
\newnumbered{question}{Question}
\newnumbered{algorithm}{Algorithm}
\newnumbered{example}{Example}
\newunnumbered{notation}{Notation} 
\usetikzlibrary{arrows}

\def\dexc{Emerald!50!}

\newcommand{\Green}[1]{{\textcolor{PineGreen}{$\mathbf{#1}$}}}
\newcommand{\red}[1]{\textcolor{red}{\mathrm{#1}}}
\newcommand{\green}[1]{\textcolor{PineGreen}{\mathrm{#1}}}

\def\hua{\mathcal}
\def\kong{\mathbb}
\def\<{\langle}
\def\>{\rangle}

\def\ZZ{\mathbb{Z}}

\def\Ind{\mathrm{Ind}}
\def\Sim{\mathrm{Sim}}
\def\Hom{\mathrm{Hom}}
\def\End{\mathrm{End}}
\def\Ext{\mathrm{Ext}}
\def\Irr{\mathrm{Irr}}

\def\Grot{K}
\def\rank{\mathrm{rank}}

\newcommand{\h}{\mathrm{\hua{H}}}            

\renewcommand{\k}{\mathbf{k}}
\renewcommand{\mod}{\mathrm{mod}}
\newcommand{\Ho}[1]{\mathrm{\bf H}_{#1}}
\newcommand{\tilt}[3]{{#1}^{#2}_{#3}}
\newcommand{\Cone}{\mathrm{Cone}}

\def\numbers{\begin{enumerate}[label=\arabic*{$^\circ$}.]}
\def\ends{\end{enumerate}}

\newcommand{\EG}{\mathrm{Eg}}       
\newcommand{\EGp}{\mathrm{Eg}^\circ}       

\newcommand{\CEG}[2]{\mathrm{CEG}_{#1}(#2)}             
\newcommand{\D}{\mathrm{\hua{D}}}

\def\nzero{\hua{H}_Q}

\newcommand{\Q}[1]{\mathcal{Q}(#1)}

\newcommand{\CY}[2]{\mathrm{CY}^{#2}(#1)}

\newcommand{\Qpe}{\widetilde{Q}}
\def\s{\mathbf{s}}
\def\v{\mathbf{v}}
\def\u{\mathbf{u}}
\def\t{\mathbf{t}}
\def\w{\mathbf{w}}
\def\gm{\w}
\newcommand{\Inv}{\mathrm{Inv}}
\newcommand{\Des}{\mathrm{Des}}
\newcommand{\Cov}{\mathrm{Cov}}

\newcommand{\nc}{\mathrm{nc}}
\newcommand{\supp}{\mathrm{supp}}
\newcommand{\add}{\mathrm{add}}
\newcommand{\VR}[1]{\mathrm{V^{r}}(#1)}
\newcommand{\Path}{\mathrm{P}}
\newcommand{\EGq}{\mathrm{Eg}_Q}
\renewcommand{\dim}{\mathrm{\underline{dim}}}
\newcommand{\wide}{\hua{W}}
\newcommand{\res}{\mathrm{res}}

\title{C-sortable words as green mutation sequences}
\author{Yu Qiu}
\classno{ 05E15, 16G20 (primary), 18E30 (secondary).}

\begin{document}\maketitle
\begin{abstract}
Let $Q$ be an acyclic quiver and $\s$ be a sequence with elements
in the vertex set $Q_0$.
We describe an induced sequence of simple (backward) tilting
in the bounded derived category $\D(Q)$, starting from the standard heart $\h_Q=\mod \k Q$
and ending at another heart $\h_\s$ in $\D(Q)$.
Then we show that $\s$ is a green mutation sequence if and only if
every heart in this simple tilting sequence is greater than or equal to $\h_Q[-1]$;
it is maximal if and only if $\h_\s=\h_Q[-1]$.
This provides a categorical way to understand green mutations.
Further, fix a Coxeter element $c$ in the Coxeter group $W_Q$ of $Q$,
which is admissible with respect to the orientation of $Q$.
We prove that the sequence $\widetilde{\gm}$
induced by a $c$-sortable word $\w$ is a green mutation sequence.
As a consequence, we obtain a bijection between $c$-sortable words
and finite torsion classes in $\h_Q$.
As byproducts, the interpretations of inversions, descents and cover reflections of
a $c$-sortable word $\w$ are given in terms of the combinatorics of green mutations.

\vskip .3cm
{\parindent =0pt \it Key words:} Coxeter group, $c$-sortable word,
quiver mutation, cluster theory, tilting
\end{abstract}

\section{Introduction}
Cluster algebras were invented by Fomin-Zelevinsky,
attempting to understand total positivity in algebraic groups and canonical bases in quantum groups.
It has been widely studied during the last decade due to
its connection to many areas in mathematics
(cf. the introduction survey \cite{Kel12}).
The crucial combinatorial ingredient in the cluster theory is quiver mutation,
which leads to the categorification of cluster algebras via quiver representation theory
due to Buan-Marsh-Reineke-Reiten-Todorov.

Recently, Keller introduced green (quiver) mutation (Definition~\ref{def:green}).
The point is to consider the principal extension $\Qpe$ of the initial quiver $Q$
by adding certain frozen vertices (and arrows).
The arrows that are incident at these frozen vertices indicate
whether an unfrozen vertex is green or red.
The extra rule is that one can only mutate a green vertex.
Using the extra combinatorial information in green mutation,
he obtained results concerning Kontsevich-Soibelman's noncommutative Donaldson-Thomas invariant
via quantum cluster algebras.
Another motivation for studying green mutation sequences comes from
theoretical physics where they yield the complete spectrum of BPS states, cf. \cite{CCV}.
Keller gave the combinatorial definition of a mutation sequence for a quiver to be green.
Such a sequence is maximal if it can not be extended to a longer green mutation sequence.
Inspired by Keller \cite{Kel11} and Nagao \cite{Nag10},
King-Qiu \cite{KQ11} studied the exchange graphs of hearts in various categories,
with applications to stability conditions and quantum dilogarithm identities in \cite{Qiu11}.
We will mainly use the techniques from \cite{KQ11} to study green mutations.

Our first aim in this paper is to interpret
Keller's green mutation (cf. Definition~\ref{def:green}) in terms of (simple) tilting of hearts.
More precisely, a green mutation sequence $\s\colon Q\rightsquigarrow Q_\s$ induces
a sequence $\Path(\s)$ of simple \emph{backward} tilting (see \S~2.4)
from the canonical heart $$\h_Q=\mod \k Q$$
to another heart $\h_\s$, in the bounded derived category $\D^b(Q)$.
Then we can describe Keller's results about green mutations via this heart $\h_\s$.
Here is a summary of categorical interpretation of green mutations in \S~\ref{sec:keller}.

\begin{theorem}\label{thm:0.1}
Let $Q$ be an acyclic quiver.
\begin{itemize}
\item   A sequence $\s$ is a green mutation sequence if and only if
$\h\geq\h_Q[-1]$ for any $\h$ in the sequence $\Path(\s)$ of tiltings.
In particular, there will be a torsion class $\hua{T}_\s$ in $\h_Q$
such that $\h_\s$ is the backward tilt of $\h_Q$ w.r.t. $\hua{T}_\s$.
\item   A vertex $j$ in $Q_\s$ for some green mutation sequence $\s\colon Q\rightsquigarrow Q_\s$
is green if and only if the corresponding simple $S_j^{\s}$ in $\h_\s$ is in $\h_Q$ and
it is not green (i.e. red)
if and only if $S_j^{\s}$ is in $\h_Q[-1]$.
We will call a simple green/red if the corresponding vertex is green/red.
\item   A green sequence $\s$ is maximal if and only if $\h_\s=\h_Q[-1]$.
Hence the mutated quiver $Q_\s$ is isomorphic to the original quiver $Q$.
\item   There is also an associated wide subcategory $\wide_{\s}$
(that is, an exact abelian subcategory closed under extensions),
generated by the shift (one) of the red simples in $\h_\s$.
\end{itemize}
\end{theorem}

Our second focus is on $c$-sortable words (where $c$ stands for Coxeter element),
introduced by Reading \cite{R07},
who showed bijections between $c$-sortable words, c-clusters and noncrossing partitions in the Dynkin case.
Ingalls-Thomas extended Reading's result in the direction of representation theory
and gave bijections between many sets (see \cite{IT09}).
The bijection between $c$-sortable words and finite torsion classes
was first generalized by Thomas
and also obtained by Amiot-Iyama-Reiten-Todorov \cite{AIRT10} via preprojective algebras.
We will interpret a $c$-sortable word as a green mutation sequence (Theorem~\ref{thm:main})
and obtain many consequences
summarized by the following theorem (see Definition~\ref{def:def} for the relevant notions).

\begin{theorem}\label{thm:0.2}
For an acyclic quiver $Q$ with an admissible Coxeter element $c$,
every $c$-sortable word $\w$ induces a green mutation sequence $\widetilde{\w}$
and we have the following bijections.
\begin{itemize}
\item   $\{$$c$-sortable word $\w\}   \overset{_{1-1}}{\longleftrightarrow}
        \{$finite torsion class $\hua{T}_{\gm}$ in $\h_Q\}$.
\item   $\{$inversion $t_{T}$ for $\w\}     \overset{_{1-1}}{\longleftrightarrow}
        \{$indecomposable $T$ in $\hua{T}_{\gm}\}$.
\item   $\{$descent $s_j$ for $\w\}    \overset{_{1-1}}{\longleftrightarrow}
        \{$red vertex $j$ for $\w\}$.
\item   $\{$cover reflection $t_{T}$ for $\w\}     \overset{_{1-1}}{\longleftrightarrow}
        \{$red simple $T$ in $\h_{\gm}\}$.
\end{itemize}
\end{theorem}
Further, if $Q$ is of Dynkin type,
the noncrossing partition $\nc_c(\w)$ associated to $\w$ can be calculated
as a product (with some order) of $\nc_c(\w)$ for $j\in \VR{\gm}$
with $\rank\nc_c(\w)=\#\VR{\gm}$,
where $\VR{\gm}$ is the set of the red vertices and $s_j^{\gm}$
is the reflection corresponding to the $j$-th simple in $\h_\w$.
Also,
the Hasse diagram of $c$-sortable words (w.r.t. the weak order)
is a supporting tree of the exchange graph $\EG_Q$ (cf. Definition~\ref{def:eg}).
Here, $\EG_Q$ is the Hasse diagram (w.r.t. the inclusion) of torsion classes in $\h_Q$.

These results give a deeper understanding of the results of Ingalls-Thomas \cite{IT09}.
Note that all our bijections are consistent with theirs,
cf. Table~\ref{table} and \cite[Table~1]{IT09}.
Also, the construction from $c$-sortable words to the green mutation sequences
is compatible with the construction of Amiot-Iyama-Reiten-Todorov \cite{AIRT10} (cf. \cite{BIRS09}).
Roughly speaking,
they use tilting theory for projectives and we use tilting theory for simples.

A warning about conventions is that our green mutation sequences for $Q$ is
the green mutation sequences for $Q^{\operatorname{op}}$ in Keller's original setting
(cf. Remark~\ref{rem:op}).

\subsection*{Acknowledgements}
I would like to thank Alastair King, Bernhard Keller, Idun Reiten and Thomas Br\"{u}stle
for helpful conversations.
Also, the work is supported by the Research Council of Norway, grant No.NFR:231000.

\section{Preliminaries}

Fix an algebraically-closed field $\k$.
Throughout this paper, $Q$ will be a finite acyclic quiver
with vertex set $Q_0=\{1,\ldots,n\}$ (unless otherwise stated).
Denote by $\k Q$ the path algebra,
$\h_Q:=\mod\k Q$ the (abelian) category of finite dimensional $\k Q$-modules and
$\D(Q):=\hua{D}^b(\h_Q)$ its bounded derived category.
Note that $\h_Q$ is abelian and $\D(Q)$ is triangulated.
We will write $\Sim\hua{A}$ for a complete set of non-isomorphic simples in
an abelian category $\hua{A}$.
Let
\[
    \Sim\h_Q=\{S_1,\ldots,S_n\},
\]
where $S_i$ is the simple $\k Q$-module corresponding to vertex $i\in Q_0$.

\subsection{Words in the Coxeter group}
Recall that the \emph{Euler form}
\[
    \<-,-\>:\kong{Z}^{Q_0}\times\kong{Z}^{Q_0}\to\kong{Z}
\]
associated to the quiver $Q$ is defined by
\[
    \<\mathbf{a}, \mathbf{b}\>
    =\sum_{i\in Q_0}a_i b_i-\sum_{(i\to j)\in Q_1}a_i b_j.
\]
Denote by $(-,-)$ the symmetrized Euler form,
i.e. $(\mathbf{a},\mathbf{b})=\<\mathbf{a},\mathbf{b}\>+\<\mathbf{b},\mathbf{a}\>$.
Moreover for $M,L\in\mod \k Q$, we have
\begin{gather}\label{eq:euler form}
    \<\dim M,\dim L\>=\dim\Hom(M,L)-\dim\Ext^1(M,L),
\end{gather}
where $\dim E\in\kong{N}^{Q_0}$ is the \emph{dimension vector} of any $E\in\mod \k Q$.
Let $V=\Grot(\k Q)\otimes\kong{R}$,
where $\Grot(\k Q)$ is the Grothendieck group of $\k Q$.
For any non-zero $v\in V$, define a \emph{reflection}
\[
    s_v(u)=u-\frac{2(v,u)}{(u,u)}v.
\]
Moreover, let $\dim E[-1]=-\dim E$ for $E\in\h_Q$ and
we will write $s_M=s_{\dim M}$ for $M\in \h_Q\sqcup\h_Q[-1]$.

The \emph{Coxeter group} $W=W_Q$ is the group of transformations generated by
the \emph{simple reflections} $s_i=s_{\dim S_i}, i\in Q_0$.
The (real) \emph{roots} in $W$ are $\{w(e_i)\mid w\in W, i\in Q_0\}$,
where $\{e_i\}$ are the idempotents;
the \emph{positive roots} are those roots which are non-negative (integral) combinations of $\{e_i\}$.
Note that the reflection of a positive root lies in $W$.
Denote by $\mathrm{T}$
the set of all the reflections of $W$,
that is, the set of all conjugates of the simple reflections of $W$.
A \emph{Coxeter element }for $W$ is the product of the simple reflections in some order.
For a Coxeter element $c=s_{\sigma_1}\ldots s_{\sigma_n}$,
we say it is \emph{admissible} with respect to the orientation of $Q$,
if there is no arrow from $\sigma_i$ to $\sigma_j$ in $Q$ for any $i>j$.

A word $\w$ in $W$ is an expression in the free monoid generated by $s_i, i\in Q_0$.
For $w\in W$, denote by $l(w)$ its \emph{length},
which is the length of the shortest word for $w$ as a product of simple reflections.
A \emph{reduced} word $\w$ for an element $w\in W$ is a word such that $\w=w$ with minimal length.
The notion of reduced word leads to the \emph{weak order} $\leq$ in $W$,
i.e. $x\leq y$ if and only if $x$ has a reduced expression which is a prefix of some reduced word for $y$.

\begin{definition}\label{def:def}
For a word $\w$ in $W_Q$, we have the following notions.
\begin{itemize}
\item   An \emph{inversion} of $\w$ is a reflection $t$ such that $l(t\w)\leq l(\w)$.
The set of inversions of $\w$ is denoted by $\Inv(w)$.
\item   A \emph{descent} of $\w$ is a simple reflection $s$ such that $l(\w s)\leq l(\w)$.
The set of descents of $\w$ is denoted by $\Des(w)$.
\item   A \emph{cover reflection} of $w$ is a reflection $t$ such that $t\w=\w s$ for some descent $s$ of $\w$.
The set of cover reflections of $\w$ is denoted by $\Cov(w)$.
\end{itemize}
\end{definition}
Similarly to normal words,
a $\mathrm{T}$-word is an expression in the free monoid generated by elements in
the set $\mathrm{T}$ of all reflections.
Denote by $l_T(w)$ its \emph{absolute length},
which is the length of the shortest word for $w$ as a product of arbitrary reflections.
So we have the notion of reduced $\mathrm{T}$-words,
which induces the \emph{absolute order} $\leq_{\mathrm{T}}$ on $W$.

The \emph{noncrossing partitions} for $W$,
(w.r.t. a Coxeter element $c$) are elements
between the trivial word $e$ and $c$, with respect to the absolute order.
The \emph{rank} of a noncrossing partition is its absolute length.

\subsection{Hearts and t-structures}
We collect some facts of tilting theory from \cite{KQ11}.
A \emph{(bounded) t-structure} on a triangulated category $\hua{D}$ is
a full subcategory $\hua{P} \subset \hua{D}$
with $\hua{P}[1] \subset \hua{P}$ satisfying the following
\begin{itemize}
\item if one defines
$\hua{P}^{\perp}=\{ G\in\hua{D}: \Hom_{\hua{D}}(F,G)=0,
  \forall F\in\hua{P}  \}$,
then for every object $E\in\hua{D}$, there is
a unique triangle $F \to E \to G\to F[1]$ in $\hua{D}$
with $F\in\hua{P}$ and $G\in\hua{P}^{\perp}$;
\item for every object $M$,
the shifts $M[k]$ are in $\hua{P}$ for $k\gg0$ and in $\hua{P}^{\perp}$ for $k\ll0$,
or equivalently,
\[
  \hua{D}= \displaystyle\bigcup_{i,j \in \ZZ} \hua{P}^\perp[i] \cap \hua{P}[j].
\]
\end{itemize}
It follows immediately that we also have
\[
  \hua{P}=\{ F\in\hua{D}: \Hom_{\hua{D}}(F,G)=0, \forall G\in\hua{P}^\perp  \}.
\]
Note that $\hua{P}^{\perp}[-1]\subset \hua{P}^{\perp}$.

The \emph{heart} of a t-structure $\hua{P}$ is the full subcategory
\[
  \h=  \hua{P}^\perp[1]\cap\hua{P}
\]
and a t-structure is uniquely determined by its heart.
More precisely, any bounded t-structure $\hua{P}$
with heart $\h$ determines a canonical filtration for each $M$ in $\hua{D}$:
\begin{equation}
\label{eq:canonfilt}
\xymatrix@C=0,5pc{
  0=M_0 \ar[rr] && M_1 \ar[dl] \ar[rr] &&  \cdots\ar[rr] && M_{m-1}
        \ar[rr] && M_m=M \ar[dl] \\
  & H_1[k_1] \ar@{-->}[ul]  && && && H_m[k_m] \ar@{-->}[ul]
  }
\end{equation}
where $H_i \in \h$ and $k_1 > \ldots > k_m$ are integers.
Moreover, the $k$-th homology of $M$, with respect to $\h$ is
\begin{gather}\label{eq:homology}
 \Ho{k}(M)=
 \begin{cases}
   H_i & \text{if $k=k_i$} \\
   0 & \text{otherwise.}
 \end{cases}
\end{gather}
Then $\hua{P}$ consists of those objects
with no (nonzero) negative homology,
$\hua{P}^\perp$ those with only negative homology
and $\h$ those with homology only in degree 0.

There is a natural partial order on hearts given by inclusion of their corresponding t-structures.
More precisely, for two hearts $\h_1$ and $\h_2$ in $\hua{D}$,
with t-structures $\hua{P}_1$ and $\hua{P}_2$,
we say
$  \h_1 \leq \h_2$
if and only if $\hua{P}_2\subset\hua{P}_1$.

\subsection{Torsion pairs and tilting}
A similar notion to t-structures in triangulated categories
is torsion pairs in abelian categories.
Tilting with respect to a torsion pair in the heart of a t-structure
provides a way to construct new t-structures.

\begin{definition}
A \emph{torsion pair} in an abelian category $\hua{C}$ is a pair of
full subcategories $\<\hua{F},\hua{T}\>$ of $\hua{C}$,
such that $\Hom(\hua{T},\hua{F})=0$ and furthermore
every object $E \in \hua{C}$ fits into a short exact sequence
$ \xymatrix@C=0.5cm{0 \ar[r] & E^{\hua{T}} \ar[r] & E \ar[r] & E^{\hua{F}} \ar[r] & 0}$
for some objects $E^{\hua{T}} \in \hua{T}$ and $E^{\hua{F}} \in \hua{F}$.
\end{definition}

\begin{proposition} [(Happel, Reiten, Smal\o)]
Let $\h$ be a heart in a triangulated category $\hua{D}$.
Suppose further that $\<\hua{F},\hua{T}\>$ is a torsion pair in $\h$.
Then the full subcategories
\[
    \h^\sharp
    =\{ E \in \hua{D}:\Ho1(E) \in \hua{F}, \Ho0(E) \in \hua{T}
        \mbox{ and } \Ho{i}(E)=0 \mbox{ otherwise} \}
\]
and
\[
    \h^\flat
    =\{ E \in \hua{D}:\Ho{0}(E) \in \hua{F}, \Ho{-1}(E) \in \hua{T}
        \mbox{ and } \Ho{i}(E)=0 \mbox{ otherwise} \}.
\]
are also hearts in $\hua{D}$.
\end{proposition}
We call $\h^\sharp$ the \emph{forward tilt} of $\h$,
with respect to the torsion pair $\<\hua{F},\hua{T}\>$,
and $\h^\flat$ the \emph{backward tilt} of $\h$.
Note that $\h^\flat=\h^\sharp[-1]$.
Furthermore, $\h^\sharp$ has a torsion pair $\<\hua{T},\hua{F}[1]\>$ and we have
\[  \hua{T}=\h\cap\h^\sharp, \quad \hua{F}=\h\cap\h^\sharp[-1].  \]
With respect to this torsion pair,
the forward and backward tilts are
$\bigl(\h^\sharp\bigr)^\sharp=\h[1]$
and $\bigl(\h^\sharp\bigr)^\flat=\h$.
Similarly $\h^\flat$ has a torsion pair $\<\hua{T}[-1],\hua{F}\>$ with
\begin{gather}\label{eq:torsion}
\hua{F}=\h\cap\h^\flat, \quad \hua{T}=\h\cap\h^\flat[1].
\end{gather}
And with respect to this torsion pair,
we have $\bigl(\h^\flat\bigr)^\sharp=\h$, $\bigl(\h^\flat\bigr)^\flat=\h[-1]$.
Recall the basic property of the partial order between a heart and its tilts as follows.

\begin{lemma}[(cf. \cite{KQ11})]
\label{lem:tiltorder}
Let $\h$ be a heart in $\hua{D}(Q)$.
Then $\h<\h[m]$ for $m>0$.
For any forward tilt $\h^\sharp$
and backward tilt $\h^\flat$, we have
\[
    \h[-1] \leq \h^\flat \leq \h \leq \h^\sharp \leq \h[1].
\]
Further, the forward tilts $\h^\sharp$ can be characterized as
precisely the hearts between $\h$ and $\h[1]$.
Similarly the backward tilts $\h^\flat$ are those between $\h[-1]$ and $\h$.
\end{lemma}


Recall that an object in an abelian category is \emph{simple}
if it has no proper subobjects, or equivalently
it is not the middle term of any (non-trivial) short exact sequence.
An object $M$ is \emph{rigid} if $\Ext^1(M,M)=0$.

\begin{definition}\label{def:simpletilt}
We say a forward tilt is \emph{simple},
if the corresponding torsion free part is generated by a
single rigid simple object $S$.
We denote the new heart by $\tilt{\h}{\sharp}{S}$.
Similarly, a backward tilt is simple
if the corresponding torsion part is generated by such a simple
and the new heart is denoted by $\tilt{\h}{\flat}{S}$.
\end{definition}

For the standard heart $\nzero$ in $\D(Q)$, an APR tilt,
which reverses all arrows at a sink/source of $Q$,
is an example of a simple forward/backward tilt.
The simple tilting leads to the notion of exchange graphs.

\begin{definition}(\cite{KQ11})\label{def:eg}
The \emph{exchange graph} $\EG\D(Q)$ of a triangulated category $\D$
is the oriented graph, whose vertices are all hearts in $\D$
and whose edges correspond to the simple \emph{\textbf{backward}} tilting between them.
We denote by $\EGp\D(Q)$ the `principal' component of $\EG\D(Q)$,
that is, the connected component containing the heart $\h_Q$.
Furthermore, denote by $\EGq$ the full subgraph of $\EGp\D(Q)$ consisting of
those hearts which are backward tilts of $\h_Q$.
Equivalently, $$\EGq=\{\h\in\EGp(Q)\mid \h_Q[-1]\leq\h\leq\h_Q \}.$$
\end{definition}

We have the following proposition which makes it possible
to tilt at any simple of any heart in $\EGq$.

\begin{proposition}[{\cite[Theorem~5.7]{KQ11}}]
Let $Q$ be an acyclic quiver.
Then every heart in $\EGp\D(Q)$ is finite and rigid
(i.e. has finitely many simples, each of which is rigid).
\end{proposition}

\begin{remark}
Unfortunately, we take a different convention to \cite{KQ11} (backward tilting instead of forward).
Thus an exchange graph in this paper has the opposite orientation
of the exchange graph there.
\end{remark}

\begin{example}\label{ex:pentagon}
Let $Q\colon=(2 \to 1)$ be a quiver of type $A_2$.
A piece of the Auslander-Reiten (AR-)quiver of $\D(Q)$ is:
    
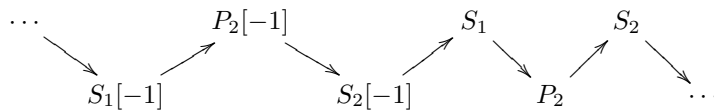
\[\xymatrix@C=1pc@R=1pc{
       \cdots\ar[dr]& &   P_2[-1] \ar[dr] &&  S_1 \ar[dr]  && S_2\ar[dr] \\
       &S_1[-1]   \ar[ur]  &&  S_2[-1]    \ar[ur]  && P_2 \ar[ur]&&\cdots}
    \]
Then $\EG_Q$ is as follows:
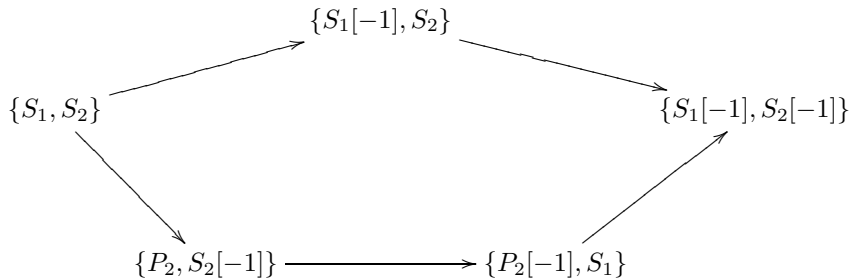
\[{\xymatrix@R=1.5pc@C=.5pc{
    &&\{S_1[-1], S_2\} \ar[drr]\\
    \{S_1, S_2\} \ar[urr] \ar[ddr] &&&& \{S_1[-1], S_2[-1]\}\\\\
    &\{P_2,S_2[-1]\} \ar[rr] && \{P_2[-1], S_1\} \ar[uur],
}}\]
where we denote a heart by the set of its simples.
\end{example}

\subsection{Simple (backward) tilting sequence}\label{sec:sst}
Let $\s=i_1\ldots i_m$ be a sequence with $i_j\in Q_0$
and we have a sequence of hearts $\h_{\s,j}$ with simples
\[\Sim\h_{\s,j}=\{S_i^{\s,j} \mid i\in Q_0\},\quad 0\leq j \leq m,\]
inductively defined as follows.
\begin{itemize}
\item   $\h_{\s,0}=\h_Q$ with $S_i^{\s,0}=S_i$ for any $i\in Q_0$.
\item   For $0\leq j\leq m-1$, we have
\[\h_{\s,{j+1}}=\tilt{(\h_{\s,j})}{ \flat }{ S^{\s,j}_j }.\]
\end{itemize}
Note that $\Sim\h_{\s,{j+1}}$ is given by the formula \cite[Proposition~5.2~(5.2)]{KQ11}
in terms of $\Sim\h_{\s,{j}}$. Thus each simple in $\Sim\h_{\s,{j+1}}$
inherits a labeling (in $Q_0=\{1,\ldots,n\}$) from the corresponding simple in $\Sim\h_{\s,{j}}$.
Define
\[
    \h_\s=\h_\s(Q)\colon=\h_{\s, m}
\]
with simples $\{S_j^\s\}$
and let $\Path(\s)$ be the path $\Path(\s)=T_m^\s \cdots T_1^\s$ as follows
\[
    \Path(\s)=\colon
        \h_Q=\h_{\s, 0} \xrightarrow{ T_1^\s }  \h_{\s,1} \xrightarrow{ T_2^\s }
        \ldots \xrightarrow{ T_m^\s } \h_{\s,m}=\h_\s,
\]
in $\EGp\D(Q)$, where $T_j^\s=S_{j}^{\s,{j-1}}$ is the $j$-th simple in $\h_{\s_{j-1}}$.
As usual, the support $\supp\Path(\s)$ of $\Path(\s)$ is the set $\{ T_1^\s,\ldots, T_m^\s \}$.


\section{Green mutation}\label{sec:keller}
In this section, we give a categorical interpretation of green mutations.
\subsection{Green quiver mutation}

\begin{definition}([Fomin-Zelevinsky])\label{def:mutation}
Let $R$ be a finite quiver without loops or $2$-cycles.
The \emph{mutation} $\mu_k$ on $R$ at vertex $k$ is a quiver $R'=\mu_k(R)$
obtained from $R$ as follows
\begin{itemize}
\item   adding an arrow $i\to j$ for any pair of arrows $i\to k$ and $k \to j$ in $R$;
\item   reversing all arrows incident with $k$;
\item   deleting as many $2$-cycles as possible.
\end{itemize}
\end{definition}

It is straightforward to see that the mutation is an involution, i.e. $\mu_k^2=id$.
A \emph{mutation sequence} $\s=i_1 \ldots i_m$ on $R$ is a sequence with $i_j\in R_0$
and we define
\[
    R_\s\colon=\mu_\s(R)=\mu_{i_m}( \mu_{i_{m-1}}(\ldots  \mu_{i_1}(R)  \ldots)).
\]
As in \S~\ref{sec:sst}, a (green) mutation sequence $\s$ induces a sequence of
simple (backward) tilting and a heart $\h_\s$.

\begin{definition}
Let $\Qpe$ be the \emph{principal extension} of $Q$,
i.e. the quiver obtained from $Q$ by adding
a new frozen vertex $i'$ and a new arrow $i'\to i$ for each vertex $i\in Q_0$.
Note that we will not mutate a quiver at a frozen vertex
and any mutated quiver of $\Qpe$ also has the frozen vertices $\{i'\}_{i\in Q_0}$.
\end{definition}
\begin{remark}\label{rem:op}
Our orientation of arrows between $i$ and $i'$ in $\Qpe$
is different from Keller's original choice in \cite{Kel11}.
The reason is that Keller considered the representations of $Q^{\operatorname{op}}$
(see \cite[p5]{Kel11})
while we prefer to deal with representations of $Q$.
Hence our green mutation sequences for $Q$ are
the green mutation sequences for $Q^{\operatorname{op}}$ in Keller's setting.
\end{remark}

\begin{definition}[(Keller \cite{Kel11})]\label{def:green}
Let $\s$ be a mutation sequence of $\Qpe$:
$$
    \Qpe=\Qpe_0\to\Qpe_1\to\cdots\to\Qpe_m=\Qpe_\s.
$$
\begin{itemize}
\item A vertex $j$ in the quiver $\Qpe_t$, $1\leq t\leq m$, is called \emph{green} if
there is no arrow from $j$ to any frozen vertex $i'$.
\item A vertex $j$ in the quiver $\Qpe_t$, $1\leq t\leq m$ is called \emph{red} if
there are no arrows to $j$ from any frozen vertex $i'$.
Let $\VR{\s}$ be the set of red vertices in $\Qpe_\s$ for $\s$.
\item $\s$ is a \emph{green mutation sequence} on $Q$ (or $\Qpe$)
if the mutation $\Qpe_{t}\to\Qpe_{t+1}$ is w.r.t. some green vertex in $\Qpe_t$,
for any $0\leq t\leq m-1$.
Such a green mutation sequence $\s$ is \emph{maximal} if $\VR{\s}=Q_0$.
\end{itemize}
\end{definition}

\begin{remark}
The original definition of a vertex being red is if the vertex is not green.
However, we will prove that our definition coincides with the original one
in Theorem~\ref{thm:keller}.
\end{remark}

\subsection{Principal extension of Ext-quivers}

Following \cite{KQ11}, we will use Ext-quivers of hearts to study green mutation.
Recall that a graded quiver is a quiver whose arrows are endowed with a $\ZZ$-grading.

\begin{definition}
\label{def:extquiv}
Let $\h$ be a finite heart in a triangulated category $\D$ with
$\mathbf{S}_{\h}=\bigoplus_{S\in\Sim\h} S$.
The Ext-quiver $\Q{\h}$ is the (positively) graded quiver
whose vertices are the simples of $\h$ and
whose graded edges correspond to a basis of
$\End^\bullet(\mathbf{S}_{\h},\mathbf{S}_{\h})$.
\end{definition}

Further, we define the \emph{CY-3 double} of a graded quiver $\hua{Q}$,
denoted by $\CY{\hua{Q}}{3}$,
to be the quiver obtained from $\hua{Q}$
by adding an arrow $T\to S$ of degree $3-k$ for each arrow $S\to T$ of degree $k$
and adding a loop of degree 3 at each vertex.
See Table~\ref{quivers} for an example of Ext-quivers and CY-3 doubling.

For the principal extension $\Qpe$ of a quiver $Q$,
consider its module category $\h_{\Qpe}$ and derived category $\D(\Qpe)$.
Since $Q$ is a subquiver of its extension $\Qpe$,
$\h_Q$ and $\D(Q)$ are subcategories of $\h_{\Qpe}$ and $\D(\Qpe)$ respectively.
For a sequence $\s$, it also induces a simple tilting sequence in $\D(\Qpe)$,
starting at $\h_{\Qpe}$ and ending at some heart, denoted by $\widetilde{\h_\s}$.

Let the set of simples in $\Sim\h_{\Qpe}-\Sim\h_Q$ be
\[\Sim\h_{Q'}\colon=\{S_i'\mid i\in Q_0\}.\]
A straightforward calculation gives
\[
    \Hom^\bullet(S_i',S_j)=\delta_{ij}\k[-1], \quad \forall i,j\in Q_0.
\]
Hence, for any $M\in\h_Q$, we have
\begin{gather}\label{eq:ijk1}
    \Hom^k(\bigoplus_{i\in Q_0} S_i',M)\neq0\quad \iff\quad k=1.
\end{gather}

\begin{lemma}\label{lem:bubian}
For any backward tilting sequence $\s$,
$\Sim\widetilde{\h_\s}=\Sim\h_{\s}\cup\Sim\h_{Q'}$.
In particular, $\Q{\h_\s}$ is a subquiver of $\Q{\widetilde{\h_\s}}$.
\end{lemma}
\begin{proof}
Use induction on the length of $\s$ starting from the trivial case when $\s=\emptyset$.
Suppose that $\s=\t j$ with $\Sim\widetilde{\h_\t}=\Sim\h_{\t}\cup\Sim\h_{Q'}$.
By \cite[Lemma~3.4]{KQ11}, we have $\h_\t\leq\h_Q$ and hence
the homology of any object in $\h_t$, with respect to $\h_Q$, lives in non-positive degrees.
Thus, any $M\in\h_\t$ admits a filtration with factors $S_i[k], i\in Q_0, k\leq0$.
As $r'$ is a source in $\Qpe$ for any $s\in Q_0$,
$S_r'$ is an injective object in $\h_{\Qpe}$ which implies that
$\Ext^1(S_i[k],S_r')=0$ for any $i\in Q_0$ and $k\leq0$.
Therefore, we have $\Ext^1(M, S_r')=0$ for any $M\in\h_\t$,
in particular, for $M=S_j^\t$.
Then applying \cite[formula~(5.2)]{KQ11} to the backward tilts
$\tilt{\h_\t}{\flat}{S_j^\t}$ and $\tilt{(\widetilde{\h_\t})}{\flat}{S_j^\t}$,
gives the required formula (since $S_r'$ never changes).
\end{proof}

\begin{definition}
\label{def:p.e.}
Given a sequence $\s$, define the \emph{principal extension} of the Ext-quiver $\Q{\h_\s}$
to be the Ext-quiver $\Q{\widetilde{\h_\s}}$ while the vertices in $\Sim\h_{Q'}$ are
the frozen vertices.
\end{definition}

From the proof of Lemma~\ref{lem:bubian}, it is straightforward to see the following.
\begin{lemma}\label{lem:source}
Every frozen vertex  $S_i'$ is a source in $\Q{\widetilde{\h_\s}}$.
\end{lemma}

\subsection{Green mutation as simple (backward) tilting}

Before we prove Keller's observations for green mutation,
we need the following result concerning the relation between
quivers for clusters and Ext-quivers for hearts.
Because the proof is technical, we leave it to the appendix.

\begin{lemma}\label{lem:KQ}
If $\widetilde{\h_\s}\in\EG_{\Qpe}$ for some sequence $\s$,
then $\Qpe_\s$ is canonically isomorphic to the degree one part of
$\CY{   \Q{\widetilde{\h_\s}}   }{3}$.
\end{lemma}
\begin{proof}
See Appendix~\ref{app}.
\end{proof}

Now we proceed to prove our first theorem.

\begin{theorem}\label{thm:keller}
Let $Q$ be an acyclic quiver and $\s$ be a green mutation sequence for $Q$.
Then we have the following.
\numbers
\item   $\h_Q[-1]\leq\h_\s\leq\h_Q$. In particular,
$\h_\s$ is the backward tilt of $\h_Q$ w.r.t. some torsion pair, denoted by $(\hua{T}_\s,\hua{F}_s)$.
\item   $\Qpe_\s$ is canonically isomorphic to the degree one part of
$\CY{   \Q{\widetilde{\h_\s}}   }{3}$.
\item   A vertex $j$ in $Q_\s$ is green (resp. red) if and only if
the corresponding simple $S_j^{\s}$ in $\h_\s$ is in $\h_Q$ (resp. $\h_Q[-1]$).
\ends
\end{theorem}
\begin{proof}
We use induction on the length of $\s$ starting with the trivial case when $l(\s)=0$.
Now suppose that the theorem holds for any green mutation sequence of length less than $m$
and consider the case when $l(\s)=m$.
Let $\s=\t j$ where $l(\t)=m-1$ and $j$ is a green vertex in $\Qpe_{\t}$.

First, the simple $S_j^\t$ corresponding to $j$ is in $\h_Q$,
by $3^\circ$ of the induction step.
This implies $1^\circ$ by \cite[Lemma~5.4, $1^\circ$]{KQ11}.

Second, as $\Sim\widetilde{\h_\s}=\Sim\h_{\s}\cup\Sim\h_{Q'}$, then
$\h_\s\in\EG_Q\,\Longleftrightarrow\,\widetilde{\h_\s}\in\EG_{\Qpe}$ follows from Lemma~\ref{lem:bubian}.
Then $2^\circ$ follows from Lemma~\ref{lem:KQ}.

Third, since $\h_Q$ is hereditary, $1^\circ$ implies
that any simple $S_j^\s\in\Sim\h_\s$ is in either $\h_Q$ or $\h_Q[-1]$.
If $S^\s_j$ is in $\h_Q$, by \eqref{eq:ijk1},
there are arrows $S_i'\to S^\s_j$ in $\Q{\widetilde{\h_\s}}$,
each of which has degree one.
Then, by $2^\circ$ any such degree one arrow
corresponds to an arrow $i'\to j$ in $\Qpe_\s$.
Therefore $j$ is green.
Similarly, if $S^\s_j$ is in $\h_Q[-1]$,
there are arrows $S_i'\to S^\s_j$ in $\Q{\widetilde{\h_\s}}$,
each of which has degree two and corresponds to
an arrow $i'\leftarrow j$ in $\Qpe_\s$.
Then $j$ is red. In all, we have $3^\circ$.
\end{proof}

For a green mutation sequence $\s$ of $Q$,
we will call a simple $S^{\s}_j\in\Sim\h_{\s}$ green/red
if the corresponding vertex $j$ is green/red in $\Qpe_\s$.

The theorem above provides a criterion for a sequence being a green mutation sequence.
Another consequence is that the quiver that corresponds to a maximal green mutation sequence
is determined as follows.

\begin{corollary}\label{cor:keller}
A sequence $\s$ is a green mutation sequence if and only if
$\h\geq\h_Q[-1]$ for any $\h\in\supp\Path(\s)$.
Further, a green mutation sequence $\s$ is maximal if and only if $\h_\s=\h_Q[-1]$.
In particular, $Q_\s\cong Q$ for a maximal green mutation sequence $\s$.
\end{corollary}
\begin{proof}
The necessity of the first statement follows from $1^\circ$ of Theorem~\ref{thm:keller}.
For the sufficiency, we only need to show that if $\t$ is a green mutation sequence
and $\s=\t j$ satisfies $\h_\s\geq\h_Q[-1]$, for some $j\in Q_0$,
then $\s$ is also a green mutation sequence.
Since $\h_\t\geq\h_Q[-1]$, by \cite[Lemma~5.4, $1^\circ$]{KQ11}
we know that $\h_\s\geq\h_Q[-1]$ implies $S_j^\t$ is in $\h_Q$.
But this means $j$ is a green vertex for $\t$, by $3^\circ$ of Theorem~\ref{thm:keller},
as required.

For the second statement,
$\s$ is maximal, if and only if $S_i^\s\in\h_Q[-1]$ for any $i\in Q_0$,
or equivalently, $\h_\s=\h_Q[-1]$.
This also implies the last statement immediately.
\end{proof}

\begin{example}\label{ex:keller}
We borrow the example of $A_2$ type green mutation sequences from Keller \cite{Kel11}
(but the orientations differ).
Figure~\ref{fig:keller} gives two different maximal green mutation sequences
($121$ and $21$) which end up being isomorphic to each other.
If we identify the isomorphic ones, we recover the pentagon in Example~\ref{ex:pentagon}.
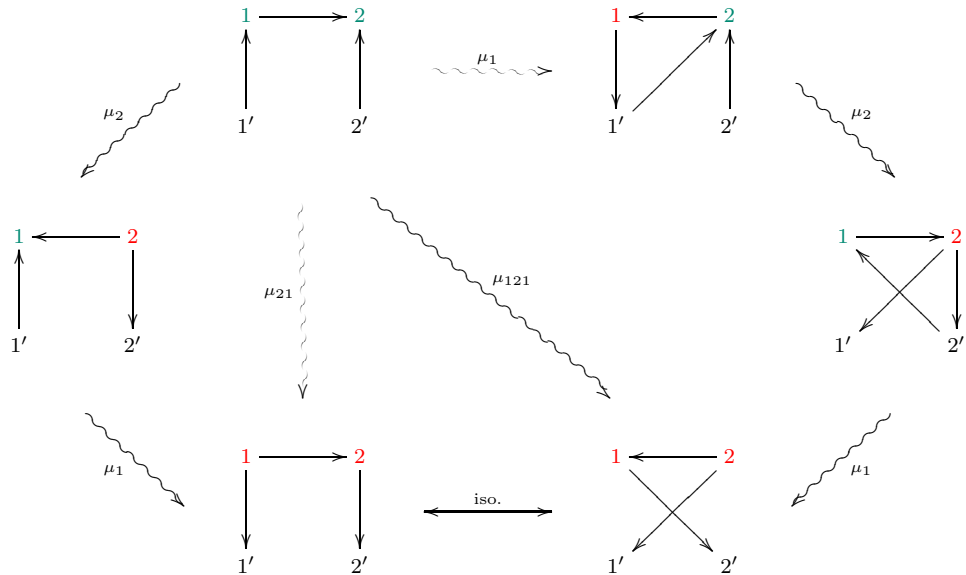
\begin{figure}[h]\centering
\[\xymatrix@C=1pc@R=1pc{
    &&&& \green{1}\ar@{<-}[dd]\ar[rr]&&    \green{2}\ar@{<-}[dd]&&&&&
            \red{1}\ar[dd]&&    \green{2}\ar@{<-}[ddll]\ar[ll]\\
    &&&\ar@{~>}[lldd]_{\mu_2}  &&&& \ar@{~>}[rrr]^{\mu_1}&&&   &&&&\ar@{~>}[rrdd]^{\mu_2}\\
    &&&& 1'  &&    2'    &&&&&   1' &&    2'\ar[uu]\\
    &&&&& \ar@{~>}[dddd]_{\mu_{21}}   &  \ar@{~>}[ddddrrrrr]^{\mu_{121}}    &&&&&&&&&&\\
    \green{1}\ar@{<-}[dd]\ar@{<-}[rr]&&    \red{2}\ar[dd]&&&&&&&&&&&&&
        \green{1}\ar@{<-}[ddrr]&&    \red{2}\ar[ddll]\ar@{<-}[ll]\\
    &&& \\
    1'  &&    2'    &&&&&&&&&&&&&   1' &&    2'\ar@{<-}[uu]\\
    &  \ar@{~>}[ddrr]_{\mu_1}  &&&&&&&&&&&&&&&  \ar@{~>}[ddll]^{\mu_1}\\
    &&&& \red{1}\ar[dd]\ar[rr]&&    \red{2}\ar[dd]&&&&&
        \red{1}&&    \red{2}\ar[ll]\\
    &&&&&&&   \ar@{<->}[rrr]^{\text{iso.}}   &&&&&&&\\
    &&&& 1'  &&    2'    &&&&&   1'\ar@{<-}[uurr] &&    2'\ar@{<-}[uull]
}\]
\caption{Two maximal green mutation sequences for an $A_2$ quiver}
\label{fig:keller}
\end{figure}
\end{example}

\subsection{Wide subcategory via red simples}\label{sec:wide}
In this section, we aim to show that the red simples, up to shift, are precisely the simples in
the wide subcategory $\wide_{\s}$ corresponding to
the torsion class $\hua{T}_{\s}$ in the sense of Ingalls-Thomas.
A \emph{wide subcategory} is an exact abelian subcategory, closed under extensions,
of some abelian category.
Further, given a finite torsion classes $\hua{T}$ in $\h_Q$,
define the corresponding wide subcategory $\wide(\hua{T})$ to be
(cf. \cite[\S~2.3]{IT09})
\begin{gather}\label{eq:defwide}
    \{ M\in\hua{T} \mid \forall (f;X\to M)\in\hua{T}, \ker(f)\in\hua{T} \}.
\end{gather}
First, we give another characterization for $\wide(\hua{T})$.

\begin{proposition}\label{pp:wide}
Let $\<\hua{F},\hua{T}\>$ be a finite torsion pair in $\h_Q$
and $\h^{\sharp}$ be the corresponding forward tilt.
Then we have
\begin{gather}\label{eq:wide}
    \Sim\wide(\hua{T})=\hua{T}\cap\Sim\h^\sharp.
\end{gather}
\end{proposition}
\begin{proof}
By \cite{IT09} and \cite{KQ11},
such a torsion pair corresponds to a cluster tilting object
(in the cluster category of $\D(Q)$) and hence $\h^\sharp$ is finite
(i.e. has finitely many simples, which generate $\h^\sharp$).
As $\h^\sharp$ admits a torsion pair $\<\hua{T},\hua{F}[1]\>$,
any of its simples is either in $\hua{T}$ or $\hua{F}[1]$.
Let $\hua{W}$ be the wide subcategory of $\h^\sharp$
generated by simples in $\hua{T}\cap\Sim\h^\sharp$.

First, for any $S\in\hua{T}\cap\Sim\h^\sharp$ and
\[
    (f:X\to S)\in\hua{T}\subset\h^\sharp,
\]
$f$ is surjective (in $\h^\sharp$) since $S$ is a simple.
Thus $\ker(f)$ is in $\hua{T}$ since $\hua{T}$ is a torsion free class in $\h^\sharp$,
which implies $S\in\wide(\hua{T})$.
Therefore $\hua{W}\subset\wide(\hua{T})$.
We claim that they coincide.

If not, let $M$ in $\wide(\hua{T})-\hua{W}$
whose simple filtration in $\h^\sharp$ (with factors in $\Sim\h^\sharp$)
has a minimal number of factors.
Let $S$ be a simple top of $M$ and then $X=\ker(M\twoheadrightarrow S)$ is in $\hua{T}$.
If $S$ is in $\hua{T}\cap\Sim\h^\sharp$,
then $X$ is in $\wide(\hua{T})-\hua{W}$ with
less simple factors, contradicting the choice of $M$.
Hence $S\in\hua{F}[1]\cap\Sim\h^\sharp$.
Then we obtain a short exact sequence
\[
    0 \to X \hookrightarrow M \twoheadrightarrow S \to 0
\]
in $\h^\sharp$ which becomes a short exact sequence
\[
    0 \to S[-1] \hookrightarrow X  \overset{f}{\twoheadrightarrow} M \to 0
\]
in $\h_Q$. But $\ker(f)=S[-1]\in\hua{F}$,
which contradicts the fact that $M$ is in $\wide(\hua{T})$ (cf. \eqref{eq:defwide}).
Therefore $\wide(\hua{T})=\hua{W}$ or \eqref{eq:wide}.
\end{proof}

We can then deduce a consequence concerning
Bridgeland's stability conditions.

\begin{definition}[Bridgeland]\label{def:stab}
A \emph{stability condition} $\sigma = (Z,\hua{P})$ on $\hua{D}$ consists of
a group homomorphism $Z:K(\hua{D}) \to \kong{C}$ called the \emph{central charge} and
full additive subcategories $\hua{P}(\varphi) \subset \hua{D}$
for each $\varphi \in \kong{R}$, satisfying the following axioms:
\numbers
\item if $0 \neq E \in \hua{P}(\varphi)$
then $Z(E) = m(E) \exp(\varphi  \pi \mathbf{i} )$ for some $m(E) \in \kong{R}_{>0}$,
\item for all
$\varphi \in \kong{R}$, $\hua{P}(\varphi+1)=\hua{P}(\varphi)[1]$,
\item if $\varphi_1>\varphi_2$ and $A_i \in \hua{P}(\varphi_i)$
then $\Hom_{\hua{D}}(A_1,A_2)=0$,
\item for each nonzero object $E \in \hua{D}$ there is a finite sequence of real numbers
$$\varphi_1 > \varphi_2 > ... > \varphi_m$$
and a collection of triangles
$$\xymatrix@C=0.8pc@R=1.4pc{
  0=E_0 \ar[rr] && E_1 \ar[dl] \ar[rr] &&   E_2 \ar[dl] \ar[rr] && ... \
  \ar[rr] && E_{m-1} \ar[rr] && E_m=E \ar[dl] \\
  & A_1 \ar@{-->}[ul]  && A_2 \ar@{-->}[ul] &&  && && A_m \ar@{-->}[ul]
},$$
with $A_j \in \hua{P}(\varphi_j)$ for all $j$.
\ends
\end{definition}

We call the collection of subcategories $\{\hua{P}(\varphi)\}$,
satisfying $2^\circ \sim 4^\circ$ in Definition~\ref{def:stab},
the \emph{slicing}.
Note that $\hua{P}(\varphi)$ is always abelian for any $\varphi\in\kong{R}$ (cf. \cite{B1})
and we call it a \emph{semistable subcategory} w.r.t. $\sigma$.

\begin{corollary}\label{cor:stab}
A finitely generated wide subcategory in $\h_Q$
is a semistable subcategory w.r.t. some Bridgeland stability condition on $\D(Q)$.
\end{corollary}
\begin{proof}
Let $\wide(\hua{T})$ be a finitely generated wide subcategory in $\h_Q$
which corresponds to the torsion pair $\<\hua{F},\hua{T}\>$.
Let $\h^{\sharp}$ be the corresponding backward tilt.

Recall that we have the following (\cite{B1}):
\begin{itemize}
\item To give a stability condition on a triangulated category $\hua{D}$
is equivalent to giving a bounded t-structure on $\hua{D}$ and
a stability function on its heart with the HN-property.
\end{itemize}
Thus, a function $Z$ from $\Sim\h^\sharp$ to the upper half plane $\kong{H}$
gives a stability condition $\sigma(Z, \h^\sharp)$ on the triangulated category $\D(Q)$.
Then choosing $Z$ as follows
\[
 Z(S)=
 \begin{cases}
   i & \text{if $S\in\Sim\h^\sharp\cap\hua{T}=\Sim\wide(\hua{T})$} \\
   0 & \text{if $S\in\Sim\h^\sharp\cap\hua{F}[1]$}
 \end{cases}
\]
will make $\wide(\hua{T})$ a semistable subcategory with respect to $\sigma(Z, \h^\sharp)$.
\end{proof}

\begin{remark}
As Bridgeland's stability conditions is a generalized version of King's $\theta$-stability condition,
Corollary~\ref{cor:stab} implies immediately
Ingalls-Thomas' result, that every wide subcategory in $\h_Q$ is a semistable subcategory
w.r.t. some $\theta$-stability condition on $\h_Q$.
\end{remark}

We return to the wide subcategory
associated to a green mutation sequence.
Let $\s$ be a green mutation sequence and
\begin{gather}\label{def:torsion}
    \hua{T}_\s=\h_Q\cap\h_\s[1]
\end{gather}
be the corresponding torsion class in $\h_Q$ in \eqref{eq:torsion}.
We will write $\hua{W}_\s$ for the wide subcategory $\wide(\hua{T}_\s)$ of $\hua{T}_\s$
as in \eqref{eq:defwide}.
Recall that $\VR{\s}$ is the set of red vertices of a green mutation sequence $\s$.
Denote by $\VR{\h_\s}=\{ S_j^\s \mid j\in\VR{\s}\}$ the set of red simples in $\h_{\s}$.

\begin{corollary}
Let $s$ be a green mutation sequence.
Then $\Sim\hua{W}_\s=\VR{\h_\s}[1]$.
\end{corollary}
\begin{proof}
By $3^\circ$ of Theorem~\ref{thm:keller}, we have
\[
    \VR{\h_\s}=\h_Q[-1]\cap\Sim\h_\s.
\]
But $\VR{\h_\s}\subset\h_\s$, hence we have
\[
    \VR{\h_\s}=(\h_Q[-1]\cap\h_\s)\cap\Sim\h_\s=\hua{T}_\s[-1]\cap\Sim\h_\s,
\]
where the second equality uses \eqref{eq:torsion}.
Noticing that $\h_\s[1]$ is the forward tilt of $\h_Q$ with respect
to the torsion class $\hua{T}_\s$, we have
\[
    \Sim\hua{W}_\s=\hua{T}_\s\cap\Sim\h_\s[1]
\]
by Proposition~\ref{pp:wide}.
Comparing with the previous equation, we have $\Sim\hua{W}_\s=\VR{\h_\s}[1]$.
\end{proof}

%

\section{C-sortable words}
In this section, we will show that it is natural to interpret a $c$-sortable word
as a green mutation sequence, which produces many interesting consequences.
Let $Q$ be an acyclic quiver with a fixed admissible Coxeter element $c$
w.r.t. the orientation of $Q$.
That means: if $c=c_{1}\cdots c_{n}$,
then there is no arrow from $i_j$ to $i_k$ in $Q$ if $j>k$.
\begin{definition}[(Reading, \cite{R07})]
For a word $a=a_1\ldots a_k$, define the support $\supp(a)$ to be $\{a_1,\ldots,a_k\}$.
Fix a Coxeter element $c=s_{\sigma_1}\ldots s_{\sigma_n}$.
A word $\w$ is called \emph{$c$-sortable} if it can be expressed as a reduced word
$\w=c^{(0)}c^{(1)}\ldots c^{(m)}$, where $c^{(i)}$ are subwords of $c$ satisfying
\begin{gather}\label{eq:c}
    \supp(c^{(m)})\subseteq\supp(c^{(1)})\subseteq\cdots\subseteq\supp(c^{(0)})\subseteq Q_0.
\end{gather}
Note that a $c$-sortable word does not depend on the reduced expression.
\end{definition}
\subsection{Main results}
Denote by $\widetilde{\gm}=i_1\ldots i_k$ the sequence induced from
a $c$-sortable word $\w=s_{i_1}\ldots s_{i_k}$.
Note that $\gm$ induces a path $\Path(\widetilde{\gm})$ and
a heart $\h_{\widetilde{\gm}}$ as in \S~\ref{sec:sst}.
We will drop the tilde of $\widetilde{\w}$ later
when it appears in the subscript or superscript.

\begin{theorem}\label{thm:main}
Let $\w$ be a $c$-sortable word. Then
\numbers
\item   $\widetilde{\gm}$ is a green mutation sequence.
\item   For any $i\in Q_0$, let $s_i^{\gm}$ be the reflection of $S^{\gm}_i$,
the $i$-th simple of $\h_{\gm}$. Then
\begin{gather}
\label{eq:main}
    s_i^{\gm} \cdot \w=\w \cdot  s_i .
\end{gather}
\item   Let the torsion class $\hua{T}_{\gm}$ be defined as in \eqref{def:torsion} and we have
$\Ind\hua{T}_{\gm}=\supp\Path({\gm})$.
\ends
\end{theorem}
\begin{proof}
We use induction on $l(\w)+\#Q_0$ starting with the trivial case $l(\w)=0$.
Suppose that the theorem holds for any $(Q, c, \w)$ with $l(\w)+\#Q_0<m$.
Now we consider the case when $l(\w)+\#Q_0=m$.
Assume that $c=s_1 c_-$ without loss of generality.

If $s_1$ is not the first letter of $\w$,
which is equivalent to $s_1\notin\supp(\w)$ by \eqref{eq:c},
then the theorem reduces to the case for $(Q_-, c_-, \w)$,
where $Q_-$ is the full subquiver with vertex set $Q_0-\{1\}$,
which is true by the inductive assumption.

Next, suppose that $s_1$ is the first letter of $\w$, so $\w=s_1 \v$ for some $\v$.
Denote by $\widetilde{\v}$ the sequence induced by $\v$.
Let $Q_+=\mu_1(Q)$, $c_+=s_1 c s_1$ and
we identify
\[
    \h_{Q_+}=\mod \k Q_+\quad\text{with}\quad\h_{s_1}=\tilt{(\h_Q)}{\flat}{S_1}
\]
via a so-called APR-tilting (reflecting the source $1$ of $Q$).
By \cite[Lemma~2.5]{R07}, $\v$ is $c_+$-sortable and hence
the theorem holds for $(Q_+, c_+, \v)$ by the inductive assumption.
Let $\v=\u s_j$, then the theorem also holds for $(Q, c, s_1 \u)$.
Let $T=S_j^\w$ be the $j$-th simple of $\h_{\w}$
and thus $T[1]$ is the $j$-th simple of $\h_{s_1 \u}$.
Using the criterion in Corollary~\ref{cor:keller} for being a green mutation sequence,
we know that
\begin{equation}
\label{eq:geq}
 \left\{
  \begin{array}{l}
    \tilt{(\h_\w)}{\sharp}{T}=\h_{s_1 \u}\geq\h_Q[-1],\\
    \h_{\w}=\h_{\w}(Q)=\h_{\v}(Q_+)\geq\h_{Q_+}[-1].
  \end{array}
 \right.
\end{equation}
If $\h_{\w}\geq\h_Q[-1]$ fails, then by comparing \eqref{eq:geq} with
\[
    \Ind\h_{Q_+}[-1]=\Ind\h_Q[-1]-\{S_1[-1]\}\cup\{S_1[-2]\},
\]
we must have $T=S_1[-2]$.
However, by formula \eqref{eq:main} for $(Q,c,s_1 \u)$ and $j\in Q_0$, we have
\[
    s_{T[1]} \cdot (s_1 \u)=(s_1 \u) \cdot  s_j.
\]
The RHS is $\w$ while the LHS equals $s_1^2 \u=\u$,
which is a contradiction to the fact that the $c$-sortable expression of $\w$ is reduced.
So $\h_{\w}\geq\h_Q[-1]$ and thus by Corollary~\ref{cor:keller}
$\widetilde{\w}$ is a green mutation sequence as $1^\circ$ required.

For $2^\circ$,
consider the influence of the APR-tilting on the dimension vectors and Coxeter group.
We know that for any $M\in\h_Q-\{S_1\}$,
the $\dim_+M$ with respect to $Q_+$ equals $s_1(\dim M)$.
Thus the reflection $t_M$ of $M$ for $Q_+$ equals $s_1 s_M s_1$
(in terms of reflections for $Q$).
In particular, the reflection $t_i^{\v}$ of $S_i^{\gm}$ for $Q_+$
equals $s_1 s_i^{\w} s_1$.
Then formula \eqref{eq:main} gives
\[
    t_i^{\v} \cdot \v=\v \cdot  s_i\,\quad
\text{or}
    \quad\,s_i^{\gm} \cdot \w=\w \cdot  s_i,
\]
as required.

Finally, we have $\Ind\hua{T}(Q)_{\gm}=\Ind\hua{T}_{\v}(Q_+)\cup\{S_1\}$
which implies $3^\circ$.
\end{proof}

\subsection{Consequences}
In this subsection,
we discuss various corollaries of Theorem~\ref{thm:main}.
First, we prove the bijection between $c$-sortable words
and finite torsion classes in $\h_Q$,
which is essentially equivalent to the result in \cite{AIRT10},
that there is a bijection between $c$-sortable words
and finite torsion-free classed in $\h_Q$.

\begin{corollary}
There is a bijection between the set of $c$-sortable words
and the set of finite torsion classes in $\h_Q$,
sending such a word $\w$ to $\hua{T}_{\gm}$.
\end{corollary}
\begin{proof}
By $3^\circ$ of Theorem~\ref{thm:main}, every torsion class $\hua{T}_{\gm}$ induced by
a $c$-sortable word $\w$ is finite.
To see two different $c$-sortable words $\w_1$ and $\w_2$ induce different
finite torsion classes, we use induction on $l(\w)$.
Then it is reduced to the case when the first letters of $\w_1$ and $\w_2$
are different.
Without loss of generality, assume that the first letter $s_1$ of $\w_1$
is on the left of the first letter $s_2$ of $\w_2$ in the expression
\[
    c=\cdots s_1 \cdots s_2 c'
\]
of the Coxeter element $c$.
Now, the sequence of simple tilting $\widetilde{\w}_2$ takes place in the full subcategory
\[\D(Q_{\res})\subset\D(Q),\]
where $Q_{\res}$ is the full subquiver of $Q$ restricted to $\supp(s_2 c')$.
Therefore the simple $S_1$ will never appear in the path $\Path(\w_2)$
which implies $\hua{T}_{\widetilde{\w}_1}\neq\hua{T}_{\widetilde{\w}_2}$
by $3^\circ$ of Theorem~\ref{thm:main}.
Therefore, we have an injection from the set of $c$-sortable words
to the set of finite torsion classes in $\h_Q$.

To finish, we need to show the surjectivity, i.e.
any finite torsion class $\hua{T}$ is equal to $\hua{T}_{\gm}$ for some $c$-sortable words.
This is again by induction for $(Q, c, \hua{T})$
on $\#\Ind\hua{T}+\#Q_0$, starting with the trivial case when $\#\Ind\hua{T}=0$.
Suppose that the surjectivity holds for any $(Q, c, \hua{T})$ with
$\#\Ind\hua{T}+\#Q_0<m$ and consider the case when $\#\Ind\hua{T}+\#Q_0=m$.
Assume that $c=s_1 c_-$ without loss of generality.

If the simple injective $S_1$ of $\h_Q$ is not in $\hua{T}$,
we claim that $\hua{T}\subset\h_{Q_-}\subset\h_Q$,
where $Q_-$ is the full subquiver with vertex set $Q_0-\{1\}$.
If so, the theorem reduces to the case for $(Q_-, c_-, \s)$, which holds by
the inductive assumption.
To see the claim, choose any $M\in\h_Q-\h_{Q_-}$.
Then $S_1$ is a simple factor of $M$ in its canonical filtration
and hence the top, since $S_1$ is injective.
Thus $\Hom(M,S)\neq0$.
But $S_1\notin\hua{T}$ implies $S_1$ is in the torsion free class corresponding to $\hua{T}$.
So $M\notin\hua{T}$, which implies $\hua{T}\subset\h_{Q_-}$ as required.

If the simple injective $S_1$ of $\h_Q$ is in $\hua{T}$,
then consider the quiver $Q_+=\mu_1(Q)$ and the torsion class
\[
    \hua{T}_+=\add\left(\Ind\hua{T}-\{S_1\}\right).
\]
Similar to the proof of Theorem~\ref{thm:main},
we know that the claim holds for $(Q_+, c+, \hua{T}_+)$,
where $c_+=s_1 c s_1$.
i.e. $\hua{T}_+=\hua{T}_{\v}$ for some $c_+$-sortable word $\v$.
But $\w=s_1 \v$ is a $c$-sortable word by \cite[Lemma~2.5]{R07} and we have
\[
    \Ind\hua{T}_{\gm}=\{S_1\}\cup\Ind\hua{T}_{\v}=\Ind\hua{T},
\]
or $\hua{T}=\hua{T}_{\gm}$, as required.
\end{proof}

Next, we investigate the length of the path $\Path(\w)$.
\begin{corollary}
Let $\w$ be a $c$-sortable word.
Then $\Path(\gm)\colon\h\rightsquigarrow\h_{\gm}$ has the maximal length
among the directed paths in $\EG_Q$ connecting $\h$ and $\h_{\gm}$.
\end{corollary}
\begin{proof}
By $4^\circ$ of Theorem~\ref{thm:main},
the number of indecomposables in $\hua{T}_{\gm}$ is exactly the length of $\Path(\gm)$.
Then the corollary follows from the fact that
each time we do a backward tilt in the sequence $\widetilde{\gm}$,
the torsion class adds at least a new indecomposable
(the simple where the tilting is at).
\end{proof}

Now, we describe the properties of a $c$-sortable word $\w$
in terms of red vertices of the corresponding green mutation sequence $\widetilde{\gm}$.
Recall that $\VR{\gm}$ is the set of red vertices of
a green mutation sequence $\widetilde{\gm}$
and $\VR{\h_{\gm}}$ the set of (red) simples in $\h_{\gm}$.
Also, see Definition~\ref{def:def} for relative notions.

\begin{corollary}
For a $c$-sortable word $\w$, the set of its
inversions, descents and cover reflections are given as follows
\begin{gather}
\label{eq:Inv}
    \Inv({\w})=\{s_T \mid T\in\supp\Path({\gm})\},\\
\label{eq:Des}
    \Des(\w)=\{s_i\mid i\in \VR{\gm}\}, \\
\label{eq:Cov}
    \Cov(\w)=\{s_T \mid T\in \VR{\h_{\gm}}\},
\end{gather}
where $\s_T$ is the reflection of $T$.
\end{corollary}
\begin{proof}
First of all, as in the proof of Theorem~\ref{thm:main} or \cite[Theorem~4.3]{IT09},
we have \eqref{eq:Inv} by inducting on $l(\w)+\#Q_0$.

For any $j\in\VR{\s}$, by $4^\circ$ of Theorem~\ref{thm:keller},
we have the corresponding simple $S_j^{\gm}$ is in $\h_Q[-1]$
and hence $S_j^{\gm}[1]$ the torsion class $\hua{T}_{\gm}$.
By formula \eqref{eq:Inv}, we know that $s_i^{\gm}$ is in $\Inv(\w)$
and hence $s_i$ is in $\Des(\w)$ by \eqref{eq:main}.

For any $j\notin\VR{\s}$, by $3^\circ$ of Theorem~\ref{thm:keller},
the corresponding simple $S_j^{\gm}$ is in $\h_Q$
but not in the torsion class $\hua{T}_{\gm}$.
Then $\dim S_j^{\gm}$ is not equal to any $\dim T, T\in\hua{T}_{\gm}$
since $\hua{T}_{\gm}$ is a simple in $\h_{\gm}\supset\hua{T}$.
Again, by formula \eqref{eq:Inv}, we know that $s_i^{\gm}$ is not in $\Inv(\w)$
and hence $s_i$ is not in $\Des(\w)$ by \eqref{eq:main}.

Therefore, we obtain \eqref{eq:Des} and \eqref{eq:Cov}.
\end{proof}

In the finite case, there are two more consequences.
The first one is about the supporting trees of the (cluster) exchange graphs.

\begin{corollary}
Let $Q$ be a Dynkin quiver.
For any $\h\in\EG_Q$, there is a unique $c$-sortable word $\w$
such that $\h=\h_{\gm}$.
Equivalently, the Hasse diagram of the $c$-sortable word $\w$ (w.r.t. the weak order)
is isomorphic to a supporting tree of the exchange graph $\EG_Q$.
\end{corollary}
\begin{proof}
The corollary follows from $3^\circ$ of Theorem~\ref{thm:main}
and the fact that any torsion class in $\h_Q$ is finite.
\end{proof}

We finish this section by showing a formula of a
$T$-reduced expression for noncrossing partitions via red vertices.
Let $\nc_c$ be Reading's map from $c$-sortable words to noncrossing partitions.
We have the following formula.

\begin{corollary}
Let $Q$ be a Dynkin quiver.
Keep the notation of Theorem~\ref{thm:main}.
Then $\nc_c(\w)$ is a product (with some order) of $\nc_c(\w)$, for $j\in \VR{\gm}$,
and the rank of $\nc_c(\w)$ equals $\#\VR{\gm}$.
\end{corollary}
\begin{proof}
The corollary follows from \eqref{eq:Cov} and Reading's map in \cite[\S~6]{R07}.
\end{proof}

\section{Example: Associahedron of dimension 3}
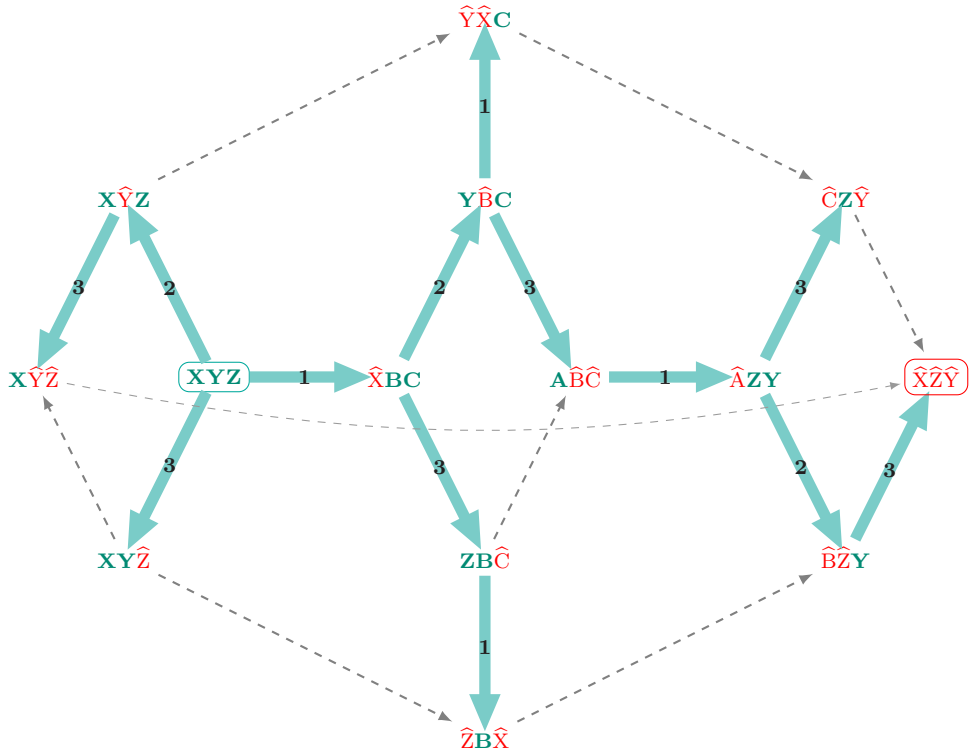
\begin{figure}[b]\centering
\begin{tikzpicture}[scale=1.2, rotate=-90, xscale=-1]
\path (0,0) node[rectangle,rounded corners,draw=white] (x1)  {\Green{X}$\red{\widehat{Y}}$$\red{\widehat{Z}}$};
\path (2,1) node[rectangle,rounded corners,draw=white] (x2)  {\Green{X}$\red{\widehat{Y}}$\Green{Z}};
\path (-2,1) node[rectangle,rounded corners,draw=white] (x3) {\Green{X}\Green{Y}$\red{\widehat{Z}}$};
\path (0,2) node[rectangle,rounded corners,draw=Emerald] (x4)  {\Green{X}\Green{Y}\Green{Z}};
\path (0,4) node[rectangle,rounded corners,draw=white] (x5)  {$\red{\widehat{X}}$\Green{B}\Green{C}};
\path (4,5) node[rectangle,rounded corners,draw=white] (x6)  {$\red{\widehat{Y}}$$\red{\widehat{X}}$\Green{C}};
\path (2,5) node[rectangle,rounded corners,draw=white] (x7)  {\Green{Y}$\red{\widehat{B}}$\Green{C}};
\path (-2,5) node[rectangle,rounded corners,draw=white] (x8) {\Green{Z}\Green{B}$\red{\widehat{C}}$};
\path (-4,5) node[rectangle,rounded corners,draw=white] (x9) {$\red{\widehat{Z}}$\Green{B}$\red{\widehat{X}}$};
\path (0,6) node[rectangle,rounded corners,draw=white]  (X)  {\Green{A}$\red{\widehat{B}}$$\red{\widehat{C}}$};
\path (0,8) node[rectangle,rounded corners,draw=white]  (X1) {$\red{\widehat{A}}$\Green{Z}\Green{Y}};
\path (2,9) node[rectangle,rounded corners,draw=white]  (X2) {$\red{\widehat{C}}$\Green{Z}$\red{\widehat{Y}}$};
\path (-2,9) node[rectangle,rounded corners,draw=white] (X3) {$\red{\widehat{B}}$$\red{\widehat{Z}}$\Green{Y}};
\path (0,10) node[rectangle,rounded corners,draw=white]  (X4) {$\textcolor{white}{\underbrace{XYZ}}$};
\path[-triangle 45, thick, line width=1.5mm]
    (x4)edge[\dexc] node[thick, Black]{\bf{2}}  (x2)
        edge[\dexc] node[thick, Black]{\bf{3}}  (x3)
        edge[\dexc] node[thick, Black]{\bf{1}}  (x5);
\path[-triangle 45, thick, line width=1.5mm]
    (x2)edge[\dexc] node[thick, Black]{\bf{3}}  (x1);
\path[-triangle 45, thick, line width=1.5mm]
    (x5)edge[\dexc] node[thick, Black]{\bf{2}}  (x7)
        edge[\dexc] node[thick, Black]{\bf{3}}  (x8);
\path[-triangle 45, thick, line width=1.5mm]
    (x7)edge[\dexc] node[thick, Black]{\bf{1}}  (x6)
        edge[\dexc] node[thick, Black]{\bf{3}}  (X);
\path[-triangle 45, thick, line width=1.5mm]
    (X1)edge[\dexc] node[thick, Black]{\bf{3}}  (X2)
        edge[\dexc] node[thick, Black]{\bf{2}}  (X3);
\path[-triangle 45, thick, line width=1.5mm]
    (X3)edge[\dexc] node[thick, Black]{\bf{3}}  (X4);
\path[-triangle 45, thick, line width=1.5mm]
    (x8)edge[\dexc] node[thick, Black]{\bf{1}}  (x9);
\path[-triangle 45, thick, line width=1.5mm]
    (X) edge[\dexc] node[thick, Black]{\bf{1}}  (X1);
\path[thick, gray, dashed] (x3) edge[->,>=latex] (x1);
\path[thick, gray, dashed] (x8) edge[->,>=latex] (X);
\path[thick, gray, dashed] (X2) edge[->,>=latex] (X4);
\path[thin,  gray!80!, dashed] (x1) edge[->,>=latex, bend left=11] (X4);
\path[thick, gray, dashed] (x3) edge[->,>=latex] (x9);
\path[thick, gray, dashed] (x9) edge[->,>=latex] (X3);
\path[thick, gray, dashed] (x2) edge[->,>=latex] (x6);
\path[thick, gray, dashed] (x6) edge[->,>=latex] (X2);
\draw (X4) node[rectangle,rounded corners,draw=red]
{$\red{\widehat{X}}$$\red{\widehat{Z}}$$\red{\widehat{Y}}$};
\end{tikzpicture}
\caption{The supporting tree of $\EG_Q$ with respect to $c=s_1 s_2 s_3$}
\label{fig:main}
\end{figure}

\begin{example}\label{ex}
Consider an $A_3$ type quiver $Q\colon2 \leftarrow 1 \rightarrow 3$
with $c=s_1s_2s_3$.
We have the Hasse diagram of $c$-sortable words below.
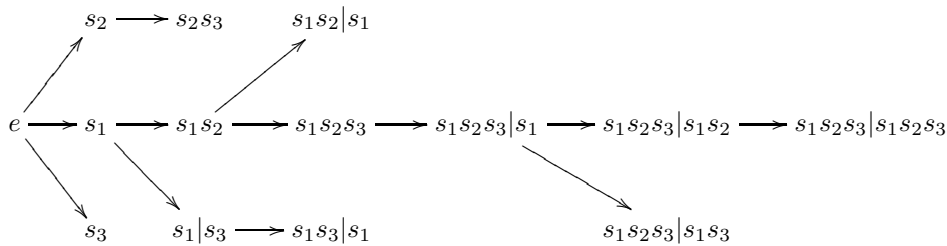
\[\xymatrix@C=1.5pc{
    &s_2\ar[r]&s_2s_3&s_1s_2|s_1\\
    e\ar[r]\ar[dr]\ar[ur]&s_1\ar[r]\ar[dr]&s_1s_2\ar[r]\ar[ur]&s_1s_2s_3\ar[r]
        &s_1s_2s_3|s_1\ar[dr]\ar[r]&s_1s_2s_3|s_1s_2\ar[r]&s_1s_2s_3|s_1s_2s_3\\
    &s_3&s_1|s_3\ar[r]&s_1s_3|s_1&&s_1s_2s_3|s_1s_3
}\]
Moreover, a piece of the AR-quiver of $\D(Q)$ is as follows
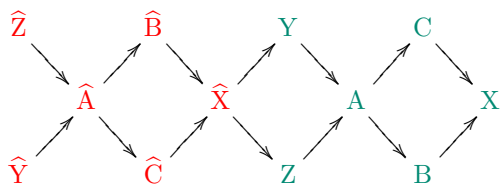
\[
\xymatrix@R=1pc@C=1pc{
    \red{\widehat{Z}}  \ar[dr] &&  \red{\widehat{B}}  \ar[dr]  && \green{Y}  \ar[dr] &&  \green{C}  \ar[dr] \\
    & \red{\widehat{A}}  \ar[ur]\ar[dr] && \red{\widehat{X}}  \ar[ur]\ar[dr] && \green{A}  \ar[ur]\ar[dr] && \green{ X}   \\
    \red{\widehat{Y}}  \ar[ur] &&  \red{\widehat{C}}  \ar[ur] && \green{Z}  \ar[ur] &&  \green{B}  \ar[ur]}
\]
where the green vertices are the indecomposables in $\h_Q$
and the red hatted ones are their shifts by $[-1]$.
Note that $\green{X},\green{Y},\green{Z}$ are the simples $S_1, S_2, S_3$ in $\h_Q$ respectively.

Figure~\ref{fig:main} is the exchange graph $\EG_Q$ (cf. \cite[Figure~1 and 4]{KQ11}).
where we denote a heart $\h_\w$ by the set of its simples $S_1^\w S_2^\w S_3^\w$ (in order).
The green edges are the green mutations in some green mutation sequences
induced from $c$-sortable words.
The number on a green edge indicates which vertex mutates.
Note that the underlying graph of Figure~\ref{fig:main} is the associahedron (of dimension 3),
i.e. the (unoriented) cluster exchange graph of type $A_3$.

Further, Table~\ref{table} is a list of correspondences between
$c$-sortable words, hearts (denoted by their simples as in the Figure~\ref{fig:main}),
descents, cover reflections, inversions and (finite) torsion classes.
Note that this table is consistent with \cite[Table 1]{IT09},
in the sense that the objects in the $j$-th row here are precisely
objects in the $j$-th row there.
\end{example}

\begin{table}[t]
\caption{Example:$A_3$}
\label{table}

\begin{tabular}{c|c|c|c|c}
\\
  \toprule
  $c$-sortable & Heart   &  Descent     &  Cover ref.
        & Torsion class \\
  word $\w$    & $\h_\w$ &  $\Des(\w)$  &  $\Cov(\w)$
        & $\hua{T}_\w$  \\  \midrule
  $s_1 s_2 s_3 | s_1 s_2 s_3$
    & ${\red{\widehat{X}}\red{\widehat{Z}}\red{\widehat{Y}}}$         & $s_1,s_2,s_3 $   & $t_X, t_Y, t_Z $   & $\green{\underline{XBCAZY}}$  \\   \midrule
  $s_1 s_2 s_3| s_1 s_2 $
    & ${\red{\widehat{B}}\red{\widehat{Z}}\green{Y}}$       & $s_1,s_2 $   & $t_B, t_Z $   & $\green{X\underline{B}C\underline{AZ}} $  \\   \midrule

  $s_1 s_2 s_3| s_1 s_3 $
    & ${\red{\widehat{C}}\green{Z}\red{\widehat{Y}}}$       & $s_1,s_3 $   & $t_C, t_Y $   & $\green{XB\underline{CAY}} $  \\   \midrule
  $s_2 s_3 $ &
    ${\green{X}\red{\widehat{Y}}\red{\widehat{Z}}}$         & $s_2,s_3 $   & $t_Y, t_Z $   & $\green{\underline{YZ}} $  \\   \midrule

  $s_1 s_2 s_3 $
    & ${\green{A}\red{\widehat{B}}\red{\widehat{C}}}$       & $s_2,s_3 $   & $t_B, t_C $   & $\green{X\underline{BC}} $  \\   \midrule
  $s_1 s_3| s_1 $
    & ${\red{\widehat{Z}}\green{B}\red{\widehat{X}}}$       & $s_1,s_3 $   & $t_Z, t_X $   & $\green{\underline{XCZ}} $  \\   \midrule

  $s_1 s_2| s_1 $
    & ${\red{\widehat{Y}}\red{\widehat{X}}\green{C}}$       & $s_1,s_2 $   & $t_Y, t_X $   & $\green{\underline{XBY}} $  \\   \midrule
  $s_2 $
    & ${\green{X}\red{\widehat{Y}}\green{Z}}$     & $s_2 $   & $t_Y $   & $\green{\underline{Y}} $  \\   \midrule

  $s_3 $
    & ${\green{X}\green{Y}\red{\widehat{Z}}}$     & $s_3 $   & $t_Z $   & $\green{\underline{Z}} $  \\   \midrule
  $s_1 s_2 s_3| s_1 $
    & ${\red{\widehat{A}}\green{Z}\green{Y}}$     & $s_1 $   & $t_A $   & $\green{XBC\underline{A}} $  \\   \midrule

  $s_1 s_2 $
    & ${\green{Y}\red{\widehat{B}}\green{C}}$     & $s_2 $   & $t_B $   & $\green{X\underline{B}} $  \\   \midrule
  $s_1 s_3 $
    & ${\green{Z}\green{B}\red{\widehat{C}}}$     & $s_3 $   & $t_C $   & $\green{X\underline{C}} $  \\   \midrule

  $s_1 $
    & ${\red{\widehat{X}}\green{B}\green{C}}$     & $s_1 $   & $t_X $   & $\green{\underline{X}} $  \\   \midrule
  $e $
    & ${\green{X}\green{Y}\green{Z}}$   & $\emptyset $   & $\emptyset $   & $\emptyset$  \\
  \bottomrule
\end{tabular}
\[ \]
\textbf{N.B.}$1\colon\;  (t_X,t_Y,t_Z,t_A,t_B,t_C)
    =(s_1,s_2,s_3,s_2 s_3 s_1 s_3 s_2, s_1 s_2 s_1, s_1 s_3 s_1)\,$.

\textbf{N.B.}$2\colon$ The underlines objects in $\hua{T}_{\gm}$
form the wide subcategory $\wide_{\gm}$.
\end{table}

\appendix
\section{Proof of Lemma~\ref{lem:KQ}}\label{app}
Recall some terminology.
For an acyclic quiver $R$,
the \emph{cluster category} $\hua{C}(R)$ of $R$ is the \emph{orbit category} $\D(R)/[-1]\circ\tau$.
A \emph{cluster tilting object} $\mathbf{P}$ in $\hua{C}(R)$ is a maximal rigid object.
Note that every such object has exactly $m$ indecomposable summands, where $m$
is the number of vertices in $R$.
One can mutate a cluster tilting object to get a new one by replacing
any one of its indecomposable summands with another unique indecomposable object in $\hua{C}(R)$.
More precisely, if $\mathbf{P}=\oplus\bigoplus_{j}P_j$, then for any $i$ we have
\[\mu_{P_i}(\mathbf{P})=P_i'\oplus\bigoplus_{j\neq i}P_i,\] where
\begin{eqnarray*}
        P_i'&=&\Cone(P_i \to \bigoplus_{j\neq i} \Irr(P_i,P_j)^*\otimes P_j)\\
            &=&\Cone(\bigoplus_{j\neq i} \Irr(P_j,P_i)\otimes P_j \to P_i)[-1].
\end{eqnarray*}
Moreover, this mutation will induce the mutation,
in the sense of Definition~\ref{def:mutation},
on the corresponding Gabriel quiver of $\mathbf{P}$.

The \emph{cluster exchange graph} $\CEG{}{R}$ of $\hua{C}(R)$ is
the unoriented graph whose vertices are cluster tilting objects
and whose edges correspond to mutations.
For instance, the cluster exchange graph of an $A_2$ quiver is a pentagon.
Buan-Thomas \cite{BT09} associated a \emph{colored quiver} $Q^C(\mathbf{P})$ to each $\mathbf{P}$,
whose degree zero part is the \emph{Gabriel quiver} $Q(\mathbf{P})$ of $\mathbf{P}$.
King-Qiu \cite{KQ11} introduced a modification of this colored quiver,
called \emph{augmented quiver} and denoted by $Q^+(\mathbf{P})$, such that
the degree one part of $Q^+(\mathbf{P})$ is the degree zero part of $Q^C(\mathbf{P})$,
i.e. $Q(\mathbf{P})$.

We need the following two lemmas.
\begin{lemma}(\cite[Corollary~5.12]{KQ11})
The underlying unoriented graph of the exchange graph $\EG_R$ (of hearts)
is canonically isomorphic to
the cluster exchange graph $\CEG{}{R}$.
\end{lemma}
\begin{lemma}(\cite[Theorem~8.7]{KQ11})
Let $\h$ be a heart in $\EG_R$ and
$\mathbf{P}$ be the corresponding cluster tilting object in $\CEG{}{R}$.
Then the CY-3 double of the Ext-quiver $\Q{\h}$ of $\h$
is isomorphic to the augmented quiver of $\mathbf{P}$
(\cite[]{KQ11}), i.e.
\[
    \CY{\Q{\h}}{3}\cong Q^+(\mathbf{P}).
\]
\end{lemma}

\begin{example}\label{ex:quivers}
We keep the notations in Example~\ref{ex}.
The left colon of quivers corresponds to the cluster tilting object
$\green{Y}\oplus\green{A}\oplus\green{B}$
and the heart $\h_1$ with simples $\{\red{\widehat{C}}, \green{Z}, \green{B}\}$;
the right colon of quivers corresponds to the cluster tilting object
$\green{Y}\oplus\red{\widehat{Z}}\oplus\green{B}$
and the heart $\h_2$ with simples $\{\red{\widehat{X}}, \red{\widehat{Z}}, \green{B}\}$.
Note that we have
\[
    \green{Y}\oplus\red{\widehat{Z}}\oplus\green{B}
    =\mu_{\green{A}}(\green{Y}\oplus\green{A}\oplus\green{B})
\]
and
\[
    \h_2=\tilt{(\h_1)}{\flat}{\green{Z}}.
\]
\begin{table}[ht]
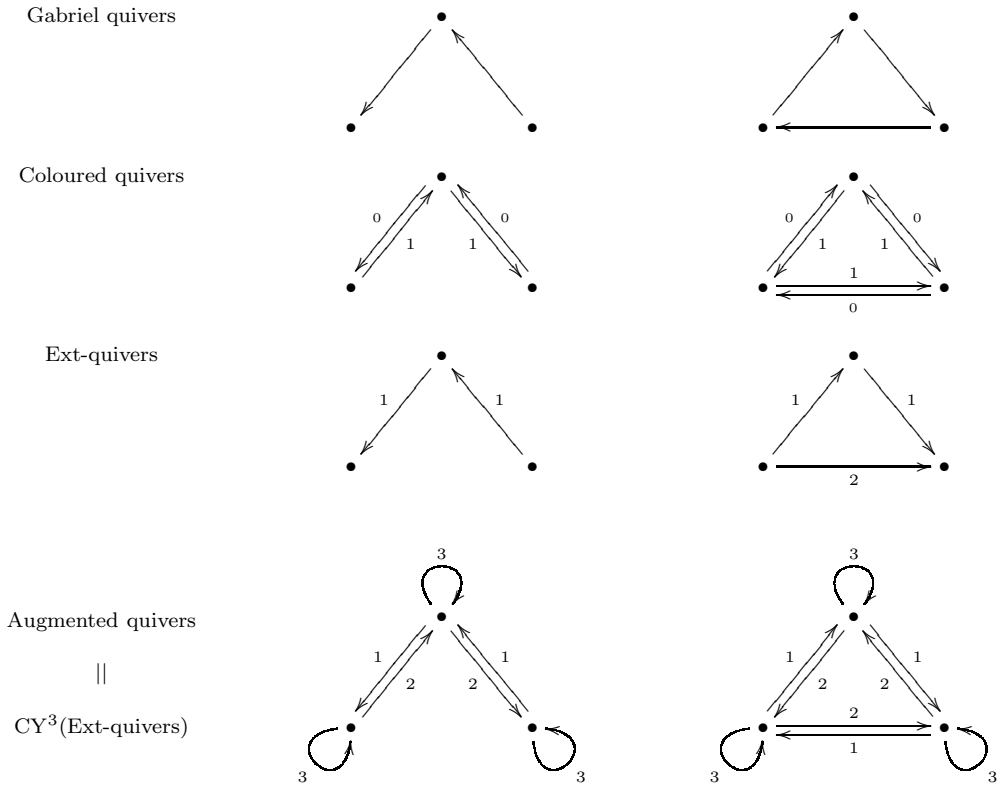
\caption{Example of various quivers}
\label{quivers}
\begin{tabular}{ccc}\\
{Gabriel quivers}&
  $\xymatrix@R=2.7pc@C=2pc{ &\bullet\ar@{<-}[dr] \\\bullet\ar@{<-}[ur]&&\bullet }$ &
  $\xymatrix@R=2.7pc@C=2pc{ &\bullet\ar@{<-}[dl] \\\bullet&&\bullet\ar@{<-}[ul]
    \ar[ll] }$
\\
  \noalign{\smallskip}
  \noalign{\smallskip}
  \noalign{\smallskip}
{Coloured quivers}&
  $\xymatrix@R=2.7pc@C=2pc{ &\bullet\ar@{<-}@<.8ex>[dr]^{_0}\ar@{<-}@<.2ex>[dl]^{1}
    \\\bullet\ar@{<-}@<.8ex>[ur]^{_0}&&\bullet\ar@{<-}@<.2ex>[ul]^{1}
  }$
  &
  $\xymatrix@R=2.7pc@C=2pc{ &\bullet\ar@<.8ex>[dr]^{_0}\ar@<.2ex>[dl]^{1}
    \\\bullet\ar@<.2ex>[rr]^{1}\ar@<.8ex>[ur]^{_0}&&\bullet\ar@<.2ex>[ul]^{1}
    \ar@<.8ex>[ll]^{_0}
  }$
\\
  \noalign{\smallskip}
  \noalign{\smallskip}
  \noalign{\smallskip}
{Ext-quivers}&
  $\xymatrix@R=2.7pc@C=2pc{ &\bullet\ar@{<-}[dr]^{1} \\\bullet\ar@{<-}[ur]^{1}&&\bullet }$ &
  $\xymatrix@R=2.7pc@C=2pc{ &\bullet\ar@{<-}[dl]_{1} \\\bullet&&\bullet\ar@{<-}[ul]_{1}
    \ar@{<-}[ll]^{2} }$ \\
\noalign{\smallskip}
\noalign{\smallskip}
\noalign{\smallskip}
\noalign{\smallskip}
\noalign{\smallskip}
\noalign{\smallskip}
{\xymatrix@R=.5pc@C=.5pc{\text{Augmented quivers}\\ \mid\mid \\\CY{\text{Ext-quivers}}{3}}}&
  $\quad\xymatrix@R=2.7pc@C=2pc{ &\bullet  \ar@(ul,ur)[]^3
    \ar@{<-}@<.8ex>[dr]^{1}\ar@{<-}@<.2ex>[dl]^{2}
      \\ \bullet\ar@(l,d)[]_3\ar@{<-}@<.8ex>[ur]^{1}&&
        \bullet\ar@(d,r)[]_3\ar@{<-}@<.2ex>[ul]^{2}
  }\quad$
  &
  $\quad\xymatrix@R=2.7pc@C=2pc{ &\bullet  \ar@(ul,ur)[]^3
    \ar@<.8ex>[dr]^{1}\ar@<.2ex>[dl]^{2}
      \\ \bullet\ar@(l,d)[]_3\ar@<.2ex>[rr]^{2}\ar@<.8ex>[ur]^{1}&&
        \bullet\ar@(d,r)[]_3\ar@<.2ex>[ul]^{2}
          \ar@<.8ex>[ll]^{1}
  }\quad$
  \\
\end{tabular}$\qquad$
\end{table}
\end{example}

Now take $R=\Qpe$ and then the heart $\widetilde{\h_\s}$ in $\EG_Q$ will correspond to
a cluster tilting set $\widetilde{\mathbf{P}}_\s$ in the cluster category $\hua{C}(\Qpe)$ with
\begin{gather}\label{eq:app}
        \CY{\Q{\h_\s}}{3}\cong Q^+(\widetilde{\mathbf{P}}_\s).
\end{gather}
Following the mutation procedure,
we deduce that $\Qpe_\s$ is isomorphic to the Gabriel quiver of the $\widetilde{\mathbf{P}}_\s$ and
hence the degree one part of \eqref{eq:app}, as required.


\affiliationone{
YU QIU\\
Institutt for matematiske fag\\
NTNU, Trondheim\\
Norway.
\email{yu.qiu@bath.edu}
}
\end{document}